\documentclass[a4paper,11pt]{article}
\usepackage{amssymb,amsmath,latexsym,amsthm}
\usepackage{color,mathrsfs}
\usepackage{url}
\usepackage{enumerate}
\usepackage[square,sort,comma,numbers]{natbib}
\usepackage{dsfont}
\usepackage{diagbox}
 \usepackage{geometry}
\usepackage{graphicx}
\usepackage{colortbl}

\usepackage{endnotes}

\renewcommand{\textbf}[1]{\begingroup\bfseries\mathversion{bold}#1\endgroup}

\usepackage{fancyhdr}
\usepackage{pstricks,pst-plot,pst-node,pstricks-add}

\newtheorem{thm}{Theorem}[section]
\newtheorem{defi}{Definition}[section]

\newtheorem{lemma}[thm]{Lemma}

\theoremstyle{definition}
\newtheorem{remark}[thm]{Remark}

\newcommand{\R}{\mathbb R}

\newcommand{\Z}{\mathbb Z}
\newcommand{\N}{\mathbb N}

\numberwithin{equation}{section}

\def\XXint#1#2#3{{\setbox0=\hbox{$#1{#2#3}{\int}$}
    \vcenter{\hbox{$#2#3$}}\kern-.5\wd0}}

\allowdisplaybreaks
\date{date}
\begin{document}
\title{On a lattice generalisation of the logarithm and a deformation of the Dedekind eta function}
\author{Laurent B\'{e}termin\\ \\
Faculty of Mathematics, University of Vienna,\\ Oskar-Morgenstern-Platz 1, 1090 Vienna, Austria\\ \texttt{laurent.betermin@univie.ac.at}. ORCID id: 0000-0003-4070-3344 }
\date\today
\maketitle

\begin{abstract}
We consider a deformation $E_{L,\Lambda}^{(m)}(it)$ of the Dedekind eta function depending on two $d$-dimensional simple lattices $(L,\Lambda)$ and two parameters $(m,t)\in (0,\infty)$, initially proposed by Terry Gannon. We show that the minimizers of the lattice theta function are the maximizers of $E_{L,\Lambda}^{(m)}(it)$ in the space of lattices with fixed density. The proof is based on the study of a lattice generalisation of the logarithm, called lattice-logarithm, also defined by Terry Gannon. We also prove that the natural logarithm is characterized by a variational problem over a class of one-dimensional lattice-logarithm.
\end{abstract}


%\tableofcontents

\section{Introduction and setting}

Many mathematical models from Physics are written in terms of special functions whose properties give fundamental information about the system (see e.g. \cite{SpecFctMP}). For example, properties of the Jacobi theta functions and Dedekind eta function defined for $\Im(\tau)>0$ by
\begin{equation}\label{def-thetaeta}
\theta_3(\tau)=\sum_{k\in \Z} e^{-i\pi k^2 \tau}, \quad \textnormal{and}\quad \eta(\tau):= q^{\frac{1}{24}}\prod_{n\in \N}(1-q^n), q=e^{2i\pi \tau},
\end{equation}
have been widely used to identify ground states of periodic systems (see e.g. \cite{Sandier_Serfaty,BeterminKnuepfer-preprint,FaulhuberExtremalDet,LuoChenWei}) and it would be impossible to write an exhaustive list of works in Mathematical Physics where they appear.

\medskip

It also has been shown of great interest to study some generalisations and deformations of these special functions which arise in more complex physical systems. In this paper, we will talk about a \textit{generalisation} when the summation appearing in the definition of the special function, which is usually on $\Z$ or a subset of it, is taken in a higher dimensional lattice. Furthermore, we will talk about \textit{deformation} when a parameter $m$, like a mass, is added to the initial function in such a way that this function is recovered in the limit $m\to 0$. While the lattice generalisation of special functions is the main topic of this paper, a typical example of deformation can be found in \cite[Section 3.1]{DBraneEta} where some deformations of the Dedekind eta function \eqref{def-thetaeta} naturally arise in Perturbative String Theory when the cylinder diagrams that determine the static interactions between pairs of Dp-branes in the type IIB plane wave background are evaluated. One of these deformations, called $\eta^{(m)}$, have been studied and generalized by Gannon in \cite{GannonEta}, defining a new object called $E_{L,\Lambda}^{(m)}(it)$ (see \eqref{Gannoneta}). The aim of this paper is to optimize $(L,\Lambda)\mapsto E_{L,\Lambda}^{(m)}(it)$ among a class of $d$-dimensional simple lattices, seen as parameters, for fixed $m,t>0$.

\subsection{Simple lattices and energies}
Before giving the precise definition of this deformation and generalisation of $\eta$, let us specify what are the spaces of lattices we are interested in. For any $d\geq 1$ and any $V>0$, we call $\mathcal{L}_d$  the set of all the simple lattices of the form
$$
L=\bigoplus_{i=1}^d \Z v_i,\quad (v_1,...,v_d) \textnormal{ basis of $\R^d$}
$$
and $\mathcal{L}_d^\circ(V)$ those with covolume $|\det(v_1,...,v_d)|=V$, which is also the volume of their unit cell. For any $f$ such that the following series is absolutely convergent, a natural question that can be asked is the optimization of
\begin{equation}\label{def-Ef}
L\mapsto E_f[L]:=\sum_{p\in L\backslash\{0\}} f(|p|^2)
\end{equation}
in $\mathcal{L}_d^\circ(V)$ for any fixed $V>0$. This problem has been studied for example in \cite{Rankin,Mont} for generalisation of the Jacobi theta function and the Riemann zeta function defined by
\begin{equation}\label{def-thetazeta}
\theta_L(\alpha):=\sum_{p\in L} e^{-\pi \alpha |p|^2},\alpha>0,\quad \textnormal{and}\quad \zeta_L(s):=\sum_{p\in L\backslash\{0\}} \frac{1}{|p|^s}, s>d,
\end{equation}
that are respectively called lattice theta function and Epstein zeta function, originally defined in \cite{KrazerPrym,Epstein1}. Both functions are of interest because of their connection with other lattice ``energies" (see e.g. \cite{BetTheta15,BetKnupfdiffuse,MorseLB} and references therein). In particular, since the energy $E_f$ of any function $f$ being the Laplace transform of a nonnegative Borel measure $\mu$, and decaying fast enough at infinity, can be written, as explained in \cite[Prop 3.1]{BetTheta15}, as
$$
E_f[L]=\int_0^{\infty} \left(\theta_L(\alpha/\pi)-1   \right)d\mu(\alpha),
$$
it follows that an optimum $L_d$ for $L\mapsto \theta_L(\alpha)$ for all $\alpha>0$ is also an optimum for $E_f$. Such functions $f$ are called \textit{completely monotone} and play an important role in the optimal point configurations community (see \cite{CohnKumar,CKMRV2Theta}). Only three global optimality results has been shown for the lattice theta function:
\begin{itemize}
\item Montgomery \cite{Mont}: If $d=2$, the triangular lattice $\Lambda_V:=\sqrt{\frac{2V}{\sqrt{3}}}\left[\Z(1,0)\oplus \Z\left(\frac{1}{2},\frac{\sqrt{3}}{2}  \right)  \right]$ is the unique minimizer of $L\mapsto \theta_L(\alpha)$ in $\mathcal{L}_2^\circ(V)$ for all $V,\alpha>0$.
\item Cohn-Kumar-Miller-Radchenko-Viazovska \cite{CKMRV2Theta}: If $d\in \{8,24\}$, $\mathsf{E}_8$ and the Leech lattice $\Lambda_{24}$ are the unique minimizers of $L\mapsto \theta_L(\alpha)$ in $\mathcal{L}_d^\circ(V)$ for all $V,\alpha>0$. More importantly, they are also the unique minimizers of this theta function among all periodic configurations of points (i.e. where there are several points per unit cell).
\end{itemize}
The results we present in this paper are restricted to the simple lattice case and, in order to have the most general results in this case, we call $\mathfrak{D}$ the set of dimensions $d$ such that $L\mapsto \theta_L(\alpha)$ has the same unique minimizer $L_d$ in $\mathcal{L}_d^\circ(V)$ for all $V,\alpha>0$, in such a way that we already know that $\{2,8,24\}\subset \mathfrak{D}\not= \emptyset$. Furthermore, as explained in \cite[p. 117]{SarStromb}, we remark that $3\not\in \mathfrak{D}\neq \N$. From the above discussion, we also notice that, for all $d\in \mathfrak{D}$, the minimizer of the lattice theta function is also the same for all the energies $E_f$ where $f$ is completely monotone.

\medskip

\subsection{Gannon's deformation of the Dedekind eta function}
The goal of this paper is to use these results in order to find the maximal possible value of the deformed general eta function defined by Gannon in \cite[Eq (3.5a)]{GannonEta} as follows. Let $V>0$, $d\geq 1$, $L,\Lambda\in \mathcal{L}_{d}^\circ(V)$ be two simple lattices and $\mathcal{P}_\Lambda$ be the set of all vectors $v\in \Lambda$ such that the only scalar multiples $\lambda v\in \Lambda$ are integer multiples. We note that $v\in \mathcal{P}_\Lambda \iff -v\in \mathcal{P}_\Lambda$ and we write $\mathcal{P}_\Lambda / \pm$ for a subset of $\mathcal{P}_\Lambda$ with only one representative of each pair of vectors $\pm v\in \Lambda$. Therefore, for any $m,t>0$, we defined
\begin{align}\label{Gannoneta}
E_{L,\Lambda}^{(m)}(it):= q^{-t^{\frac{d-1}{2}}\Delta_m(L)}\prod_{w\in \mathcal{P}_\Lambda / \pm}\prod_{v\in L} \left(1-q^{|w|(m^2+|v|^2)}   \right)^{\frac{t^{\frac{d-1}{2}}}{2|w|}}, \quad q=e^{-\pi t},
\end{align}
  where 
  \begin{align*}
  \Delta_m(L):=-\frac{1}{8\pi |L|^{1/2}}\int_0^\infty s^{\frac{d-1}{2}} e^{-\pi \frac{m^2}{s}}\left( \theta_{L^*}(s)-1 \right)ds,
  \end{align*}
  and $\theta_{L^*}$ is the lattice theta function defined by \eqref{def-thetazeta} for the dual lattice $L^*$ of $L$ defined by
$$
L^*:=\{y\in \R^d : \forall p\in L, y\cdot p\in \Z\}.
$$  
For another formula which may look more familiar to the reader, see \eqref{eq-formulaenergytotal}. This function $E_{L,\Lambda}^{(m)}$ can be viewed as a deformation of the Dedekind eta function $\eta$ defined by \eqref{def-thetaeta} in the following sense. We first remark that, for any $t>0$,
  \begin{align*}
  E_{\Z,\Z}^{(m)}(it)=:\eta^{(m)}(it),
  \end{align*}
  where $(2\pi mt)^{-1/2}\eta^{(m)}(it)\to \eta(it)$ as $m\to 0$. As explained before, the function $\eta^{(m)}(it)$ appears in a Perturbative String Theory \cite{DBraneEta}. We also notice that $\Delta_m$ is the lattice generalisation of the Casimir effect of a boson in a cylindrical worldsheet, the masslike parameter $m$ measuring how the background space-time -- where the D-branes are located -- is curved. 

\subsection{Goal of our work}
  
  In order to explore the properties of this deformation of $\eta$, we ask the following question leading to the knowledge of extremal values for possible associated physical systems:
\begin{itemize}
\item \textbf{Problem:} What are the global maximizers of $(L,\Lambda)\mapsto E_{L,\Lambda}^{(m)}(it)$ in $\mathcal{L}_d^\circ(V_1)\times\mathcal{L}_d^\circ(V_2) $ where $V_1,V_2>0$ and $m,t>0$ are fixed?
\end{itemize}  

Studying $(L,\Lambda)\mapsto E_{L,\Lambda}^{(m)}(it)$ in $\mathcal{L}_d^\circ(V_1)\times\mathcal{L}_d^\circ(V_2) $ is an interesting mathematical question by itself which encourages us to investigate the maximality properties of other lattice energies, expanding the use of known variational techniques to new objects, like the lattice-logarithm defined in the next subsection. Furthermore, even though $E_{L,\Lambda}^{(m)}(it)$ does not appear yet in any model of (Perturbative) String Theory, we hope that our results will help mathematical physicists to optimize quantities like an analog of the cylinder diagram in light-cone coordinates or the open-string one-loop diagram (see \cite[Section 3]{DBraneEta}). We also believe that this type of lattice sums could be useful in Number Theory where such $q$-products appear frequently and are sometimes related to geometrical quantities, as  $\eta$ is related to the height of the flat torus in dimension $d=2$ (see \cite{OPS}). This problem is solved in Theorem \ref{thm-main}, showing that for any $d\in \mathfrak{D}$, the minimizers of the lattice theta functions are the maximizers of $(L,\Lambda)\mapsto E_{L,\Lambda}^{(m)}(it)$.

\subsection{The lattice-logarithm}
  
As proposed by Gannon in the same paper (see \cite[p. 64]{GannonEta}), we define the lattice generalisation of the logarithm as follows, which is connected to $E_{L,\Lambda}^{(m)}(it)$ by \eqref{eq-formulaenergytotal}.
 
  \begin{defi}[Lattice-logarithm]
  For any $d\geq 1$, any $V>0$ and any simple lattice $L\in \mathcal{L}_d^\circ(V)$, we define the $L$-logarithm of any real number $x\in (0,1)$ by
\begin{align*}
\log_L(x):=-\frac{1}{2}\sum_{p\in L\backslash \{0\}} \frac{(1-x)^{|p|}}{|p|}.
\end{align*}
Furthermore, if $d=1$ and $\{t_n\}_{n\in \Z}$ is a $N$-periodic sequence of real numbers for some $N\geq 1$, i.e. such that $t_{n+N}-t_n=N$ for all $n\in \Z$, then the $\{t_n\}$-logarithm of $x\in (0,1)$ is defined by
\begin{equation*}
\log_{\{t_n\}}(x):=-\frac{1}{4N}\sum_{i=1}^N \sum_{j\in \Z\backslash \{i\}} \frac{(1-x)^{|t_j-t_i|}}{|t_j-t_i|}.
\end{equation*}
  \end{defi}
This $L$-logarithm is indeed a generalisation of the logarithm in the sense that $\log_\Z(x)=\log(x)$ for any $x\in (0,1)$, and Gannon asked the question about the properties of this object. 

\medskip

\noindent\textbf{Plan of the paper.} In the next section, we study the maximization of $L\mapsto \log_L(x)$ for any given $x\in (0,1)$ as well as the problem of maximizing a lattice energy (see \eqref{def-ELLambda}) of type $E_f$ defined as in \eqref{def-Ef}, where $f$ depends on a lattice-logarithm. The fact that the $L$-logarithm is summed over another lattice $\Lambda$ is the first example of this kind, when $L\neq \Z$. Finally, in Section \ref{sec-MaxE}, we show our main result about the maximality of lattices $\left(V_1^{\frac{1}{d}}L_d,V_2^{\frac{1}{d}}L_d\right)$ for $E_{L,\Lambda}^{(m)}(it)$ when $d\in \mathfrak{D}$ in $\mathcal{L}_d^\circ(V_1)\times\mathcal{L}_d^\circ(V_2)$.

  \section{Properties of the lattice-logarithm}  \label{sec-Llog}
  
We start by showing, in dimension $d=1$, that the natural logarithm is characterized by a maximality problem. This result also allows to construct functions, in a canonical way, that are smaller and close to $\log(x)$.

  \begin{thm}[Characterization of the logarithm]
For any $x\in (0,1)$, we have, for any $N$,
\begin{align*}
\log(x)=\max_{\{t_n\}_{n\in \Z}} \left\{\log_{\{t_n\}}(x): \forall n\in \Z, t_{n+N}-t_n=N  \right\},
\end{align*}
and this maximum is achieved if and only if $\{t_n\}=\Z+a$ for some $a\in \R$.
\end{thm}
\begin{proof} Let $x\in (0,1)$ and, for any $r>0$, $\phi_x(r):=-2\ell_x(r)$ where $\ell_x(r):=\frac{-(1-x)^r}{2r}$.   We first prove that $\phi_x$ is strictly convex, i.e. $\phi_x''(r)>0$ for any $r>0$. Indeed, we know that 
\begin{equation*}
\frac{(1-x)^r}{r}=\frac{e^{r\log(1-x) }}{r}=\int_{-\log(1-x)}^{\infty} e^{-tr}dt,
\end{equation*}
and it follows that $\phi_x''(r)=\int_{-\log(1-x)}^{\infty} t^2 e^{-tr}dt>0$ on $(0,\infty)$.

\noindent Thus, by Ventevogel's optimality result \cite[Theorem 1]{Ventevogel1}, we get, for any $N\in \N$ and any $N$-periodic sequence $\{t_n\}$,
\begin{equation*}
\frac{1}{2N}\sum_{i=1}^N \sum_{j\in \Z\backslash \{i\}} \phi_x(|t_j-t_i|)\geq \sum_{k=1}^{\infty} \phi_x(k).
\end{equation*}
Thus, we get
\begin{equation*}
\log_{\{t_n\}}(x)\leq \sum_{k=1}^{\infty} \ell_x(k)=-\frac{1}{2}\sum_{k=1}^{\infty}\frac{(1-x)^{k}}{k}=\log_{\Z}(x)=\log(x),
\end{equation*}
and our result is proved because $\{t_n\}=\Z+a$ are the only sequences satisfying the equality case in Jensen's inequality.
\end{proof}
\begin{remark}\label{rmk-expL}
One might be tempted to think that this result is true for any kind of lattice generalisation of a classical function, but one would be wrong. Indeed, we can define a $\{t_n\}$-exponential of $x>0$ by
\begin{equation}\label{def-expL}
\exp_{\{t_n\}}(x):=\frac{1}{2N}\sum_{i=1}^N \sum_{j\in \Z\backslash \{i\}} \frac{x^{|t_j-t_i|}}{\Gamma(|t_j-t_i|+1)},
\end{equation}
and ask the same question of the global optimality of $\{t_n\}=\Z+a$. It turns out that the function $\psi_x(r)=\frac{x^r}{\Gamma(r+1)}$ is not convex for some values of the parameter, for example $x=e$, so our method cannot be used. Furthermore, for $x=e$, the function is not even decreasing on $(0,\infty)$, which makes it a bad candidate for the optimality of the equidistant configuration.
\end{remark}

In higher dimension $d\geq 2$, we restrict our study to the simple lattice case and we have the following result connecting the minimizers of the lattice theta function and the one of $L\mapsto \log_L(x)$ for any given real number $x\in (0,1)$.

\begin{thm}[Maximizer of the $L$-logarithm]\label{thm-maxlogL}
Let $d\in \mathfrak{D}$, $x\in(0,1)$ and $V>0$. Then, the unique minimizer $V^{\frac{1}{d}}L_d$ of the lattice theta function is the unique maximizer of $L\mapsto \log_L(x)$ in $\mathcal{L}_d^\circ(V)$.
\end{thm}
\begin{proof}
It is sufficient to show that $\phi_x(r):=\frac{(1-x)^{\sqrt{r}}}{\sqrt{r}}$ is a completely monotone function and to use \cite[Prop 3.1]{BetTheta15} already mentioned in the introduction in order to conclude. Writing $\phi_x(r)=\frac{e^{a\sqrt{r}}}{\sqrt{r}}$ and remarking that $a:=\log(1-x)<0$, $f_x$ is therefore completely monotone as a product of completely monotone functions (see \cite{ComplMonotonic}).
\end{proof}

Since $\log_L$ is a function by itself on $(0,1)$, we can use it to define a new lattice sum where the $L$-logarithm is an interacting potential. The problem of finding optimal point configurations for logarithmic interaction is indeed of great interest in dimension $d=2$ \cite{SerfatyCoulomb} (where $-\log$ is the Coulomb potential) and on the $(d-1)$-dimensional sphere \cite{BrauchGrab}. We therefore obtain an energy depending on two simple lattices that we can maximize as follows.

\begin{thm}[The $L$-logarithm viewed as interacting potential]\label{thm-logpot} Let $d\in \mathfrak{D}$ and $V_1,V_2>0$. For any simple lattices $L\in \mathcal{L}_d^\circ(V_1)$ and $\Lambda\in \mathcal{L}_d^\circ(V_2)$, and for any function $f$ such that $f'$ is completely monotone, we define
\begin{equation}\label{def-ELLambda}
\mathcal{E}_{f}[L,\Lambda]:=\sum_{q\in \Lambda} \log_L(1-e^{-f(|q|^2)}). 
\end{equation}
Then, for any fixed $L$ and $f$, $V_2^{\frac{1}{d}}L_d$ is the unique maximizer of $\Lambda \mapsto \mathcal{E}_{f}[L,\Lambda]$ in $\mathcal{L}_d^\circ(V_2)$. Furthermore,  $\left(V_1^{\frac{1}{d}}L_d,V_2^{\frac{1}{d}}L_d\right)$ is the unique pair of maximizers of $ \mathcal{E}_{f}$ in $\mathcal{L}_d^\circ(V_1)\times \mathcal{L}_d^\circ(V_2)$.
\end{thm}
\begin{proof}
We have, by absolute convergence,
\begin{equation}\label{eq-energyrewritten}
\mathcal{E}_{f}[L,\Lambda]=-\frac{1}{2}\sum_{q\in \Lambda} \sum_{p\in L\backslash\{0\}} \frac{e^{-|p|f(|q|^2)}}{|p|}=-\frac{1}{2}\sum_{p\in L\backslash \{0\}} \frac{1}{|p|}\left(\sum_{q\in \Lambda} e^{-|p| f(|q|^2)}\right),
\end{equation}
and it is clear that the energy is well-defined. Since $f$ is such that $f'$ is a completely monotone function and $r\mapsto e^{-|p|r}$ is also completely monotone for any $p\in L\backslash \{0\}$, it follows from \cite[Thm 2]{ComplMonotonic} that  $r\mapsto e^{-|p| f(r)}$ is completely monotone. Therefore, by \cite[Prop 3.1]{BetTheta15}, $V_2^{\frac{1}{d}}L_d$ is the unique minimizer of $\Lambda \mapsto \sum_{q\in \Lambda} e^{-|p| f(|q|^2)}$ in $\mathcal{L}_d^\circ(V_2)$ and the first part of the theorem is proved. The second part follows from the first part by using \eqref{eq-energyrewritten} and Theorem \ref{thm-maxlogL}.
\end{proof}
\begin{remark}
As already mentioned in the introduction, it is, as far as we know, the first time that such energy with an interaction potential depending itself on a lattice is studied. It might be interesting to studied a lattice generalisation of the lattice theta function, using the $L$-exponential defined by \eqref{def-expL} and see if one can again derive optimality results of this kind. The fact that the  associated potential $\psi_x$ defined in Remark \ref{rmk-expL} is not convex might be a problem but also an interesting direction to explore.
\end{remark}

  \section{Maximization of $E_{L,\Lambda}^{(m)}(it)$ among simple lattices} \label{sec-MaxE}

It is straightforward to obtain the following formula that connects $E_{L,\Lambda}^{(m)}(it)$ to the lattice-logarithm.
  
  \begin{lemma}[\cite{GannonEta}]
  For any $t,m>0$ and any simples lattices $L,\Lambda\in \mathcal{L}_d$, we have
  \begin{align}\label{eq-formulaenergytotal}
  \log E_{L,\Lambda}^{(m)}(it)=-\frac{t^{\frac{d+1}{2}}}{8 |L|^{1/2}}\int_0^{\infty} s^{\frac{d-1}{2}}e^{-\pi\frac{m^2}{s}}\left( \theta_{L^*}(s)-1 \right)ds+t^{\frac{d-1}{2}}\sum_{p\in L} \log_\Lambda(1-q^{m^2 + |p|^2}).
  \end{align}
  \end{lemma}

Therefore, we can easily derive the maximum of $(L,\Lambda)\mapsto E_{L,\Lambda}^{(m)}(it)$ among simple lattices.

\begin{thm}[Maximizer of $E_{L,\Lambda}^{(m)}(it)$]\label{thm-main}
Let $d\in \mathfrak{D}$, $t,m>0$ and $V_1,V_2>0$. Then $\left(V_1^{\frac{1}{d}}L_d,V_2^{\frac{1}{d}}L_d\right)$ is the unique maximizer of $(L,\Lambda)\mapsto  E_{L,\Lambda}^{(m)}(it)$ in $\mathcal{L}_d^\circ(V_1)\times\mathcal{L}_d^\circ(V_2)$.
\end{thm}
\begin{proof}
We use the previous formula \eqref{eq-formulaenergytotal}. Since $d\in \mathfrak{D}$, $s^{\frac{d-1}{2}}e^{-\pi\frac{m^2}{s}}>0$ for any $s>0$ and $L_d=L_d^*$ as a simple consequence of the Poisson Summation Formula, we know that 
$$
L\mapsto \int_0^{\infty} s^{\frac{d-1}{2}}e^{-\pi\frac{m^2}{s}}\left( \theta_{L^*}(s)-1 \right)ds
$$
is minimized by $V_1^{\frac{1}{d}}L_d$ in $\mathcal{L}_d^\circ(V_1)$. Therefore the first term is maximized by $V_1^{\frac{1}{d}}L_d$. For the second term, we use Theorem \ref{thm-logpot} with $f(r)=\pi t(m^2+r)$ and we remark that $f'(r)=\pi t$ is a completely monotone function, which implies that the second term is maximized by $\left(V_1^{\frac{1}{d}}L_d,V_2^{\frac{1}{d}}L_d\right)$, which concludes the proof.
\end{proof}

\medskip

\noindent \textbf{Acknowledgement:} I acknowledge support
from VILLUM FONDEN via the QMATH Centre of Excellence (grant No. 10059).

{\small
\bibliographystyle{plain} 
\bibliography{Bibliogeneral}}
\end{document}